\documentclass[leqno,11pt]{article}
\usepackage{amssymb,amsmath,mathabx}

\usepackage{pgf,ifpdf,graphics,xcolor}

\usepackage{amsfonts,amssymb,euscript,epsfig,color}
\usepackage[all]{xy}

\newtheorem{theo}{Theorem}[section]
\newtheorem{defi}[theo]{Definition}

\newtheorem{lem}[theo]{Lemma}
\newtheorem{prop}[theo]{Proposition}
\newtheorem{rem}[theo]{Remark}
\newtheorem{coro}[theo]{Corollary}
\newtheorem{exam}[theo]{Example}

\newenvironment{proof}{{\bf Proof.}}

\newcommand{\agot}{\ensuremath{\mathfrak{a}}}

\newcommand{\hgot}{\ensuremath{\mathfrak{h}}}
\newcommand{\kgot}{\ensuremath{\mathfrak{k}}}

\newcommand{\ngot}{\ensuremath{\mathfrak{n}}}

\newcommand{\qgot}{\ensuremath{\mathfrak{q}}}

\newcommand{\sgot}{\ensuremath{\mathfrak{s}}}
\newcommand{\tgot}{\ensuremath{\mathfrak{t}}}

\newcommand{\zgot}{\ensuremath{\mathfrak{z}}}

\newcommand{\Acal}{\ensuremath{\mathcal{A}}}

\newcommand{\Ccal}{\ensuremath{\mathcal{C}}}

\newcommand{\Fcal}{\ensuremath{\mathcal{F}}}
\newcommand{\Hcal}{\ensuremath{\mathcal{H}}}

\newcommand{\Ocal}{\ensuremath{\mathcal{O}}}
\newcommand{\Pcal}{\ensuremath{\mathcal{P}}}

\newcommand{\Scal}{\ensuremath{\mathcal{S}}}

\newcommand{\Z}{\ensuremath{\mathbb{Z}}}
\newcommand{\C}{\ensuremath{\mathbb{C}}}

\newcommand{\R}{\ensuremath{\mathbb{R}}}

\newcommand{\mm}{\ensuremath{\hbox{\rm m}}}

\newcommand{\tr}{\ensuremath{\hbox{\bf Tr}}}
\newcommand{\ntr}{\ensuremath{\hbox{\bf nTr}}}

\newcommand{\QS}{\ensuremath{\mathrm{Q}^{\mathrm{spin}}}}

\newcommand{\spin}{\ensuremath{\rm Spin}}
\newcommand{\spinc}{\ensuremath{{\rm Spin}^{c}}}

\newcommand{\CW}{{\mathfrak t^*_{\geq 0}}}

\def \wK {\widehat{K}}

\def \what {\widehat}

\def \ad {{\rm ad}}

\begin{document}

\title{Admissible coadjoint orbits for compact Lie groups.}

\author{Paul-Emile PARADAN\footnote{Institut Montpelli\'erain Alexander Grothendieck, CNRS UMR 5149,
Universit\'e de Montpellier, \texttt{paradan@math.univ-montp2.fr}}
\hspace{1mm}  and
Mich\`ele VERGNE\footnote{Institut de Math\'ematiques de Jussieu, CNRS UMR 7586,
Universit\'e Paris 7, \texttt{michele.vergne@imj-prg.fr}}}


\maketitle

{\small
\tableofcontents}



\section{Introduction}

Let $K$ be a real Lie group with Lie algebra $\mathfrak k$.
The notion of admissible coadjoint orbit was introduced by M. Duflo \cite{du} in order to study the unitary dual of $K$.
This notion is  also important in the work of D.Vogan (see \cite{vogan}).
Indeed, quoting him, ``the orbit method provides some light in a very dark room".
In this article, we restrict ourselves to the case where $K$ is a compact connected Lie group.
In this case, everything seems clear, admissible orbits are very easy to describe as well as $\hat K$. Nevertheless,
 we discovered  some remarkable properties of admissible coadjoint orbits
 which (we believe) were unnoticed before.
 Furthermore, as we briefly explain below, these properties are fundamental in the study of the equivariant index of twisted Dirac operators on general $K$-manifolds. It thus became necessary  to return to this subject.
Let us describe our motivations, and some of our results.

Choose a $K$-invariant scalar product on $\kgot^*$, a Cartan subalgebra, and a system of positive roots and
 denote by $\|\rho^K\|$ the norm of $\rho^K=\frac{1}{2}\sum_{\alpha>0}\alpha$.

Coadjoint orbits can be grouped in different subsets  that we call Dixmier  sheets.
If a coadjoint orbit $K\xi$ varies in a Dixmier sheet, the stabilizer subgroup $K_\xi$ of $\xi\in \kgot^*$ remains in a fixed conjugacy class.
 The coadjoint orbits of maximal dimension form the regular sheet.
Furthermore,  there is a map  $\QS_K$ (see Definition \ref{defi:QS})
associating to an admissible coadjoint orbit $\Pcal$  a virtual representation  $\QS_K(\Pcal)$ of $K$.
By this correspondence, regular  admissible  coadjoint orbits parameterize the set $\wK$ of classes of unitary irreducible representations of $K$. For example, the coadjoint orbit of $\rho^K$ is regular admissible
and parameterize the trivial representation of $K$.

 We  associate to a Dixmier sheet $S$ a  subset $I_S$ of $\wK$, a positive number $\|\rho_S\|$ and  a conjugacy class $\sgot_S$  of
 semi-simple subalgebras of $\kgot$:
 the set $I_S$ is the set of irreducible representations of $K$ of the form
 $\QS_K(\Pcal)$ where $\Pcal$ is an admissible orbit belonging  to  $S$,
 $\sgot_S$ is  the conjugacy class of the semi-simple part of the infinitesimal stabilizer $\kgot_\xi$
 for any orbit $K\xi$ in $S$, and  $\|\rho_S\|$ is the norm of the $\rho$-element for the subgroup $K_\xi$ of $\kgot$
(thus $\|\rho_S\|=0$ exactly when $S$ is the regular sheet).

 Let $\Scal_\kgot$  be the set of  conjugacy classes of the subalgebras $\sgot_S$ where $S$ varies over all Dixmier sheets.
The  set $\Scal_\kgot$  and the particular subset $I_S$ of $\hat K$ associated to a Diximier sheet $S$  occur in the following  theorem, that we prove in \cite{pep-vergne:spin2} and announced in \cite{pep-vergne:cras}.

Assume $K$ is  acting on a connected  $\spin$ manifold $M$ (or more generally a $\spinc$ manifold).
Let $L$ be a $K$-equivariant line bundle on $M$.
 Consider  the twisted Dirac operator $D_L$  and its equivariant index $Q(M,L)$.
Then the following holds.

 $\bullet$ If the semi-simple part $\sgot$ of the generic infinitesimal stabilizer $\kgot_M$   of the action of $K$ in $M$
 does not belong to $\Scal_\kgot$,   then $Q(M,L)=\{0\}$ for any equivariant line bundle $L$ on $M$.

$\bullet$ If $\sgot$ is equal to $\sgot_S$,
 then $Q(M,L)$ is a virtual sum of irreducible representations of $K$ belonging to $I_S$.

\bigskip

The proof of this  result uses  an inequality on distance between coadjoint orbits, that we call the magical inequality.
We prove this inequality in this article (Theorem \ref{prop:infernal}).
 In this introduction, let us just state two of its consequences.

\begin{theo}\label{theoinfernalintro}
Let $\Ocal$ be a regular admissible coadjoint orbit, and let $\Pcal$ be a coadjoint orbit in a Dixmier sheet $S$.
Then the distance between $\Ocal$ and $\Pcal$ is greater or equal to  $\|\rho_S\|.$
\end{theo}

Assume that $\Pcal$ is an admissible coadjoint orbit.
We prove in Theorem \ref{theo:notregular} that $\QS_K(\Pcal)$ is either $0$, or irreducible. If $\QS_K(\Pcal)$ is irreducible, then
$\QS_K(\Pcal)=\QS_K(\Ocal)$ for a unique admissible regular  orbit $\Ocal$.  We then say that $\Pcal$ is an ancestor of $\Ocal$.
We prove the following theorem (Theorem \ref{ancestorsanddistances}).
\begin{theo}\label{theoancestorsintro}
Let $\Ocal$ be a regular admissible orbit.
The ancestors of $\Ocal$ in a Dixmier sheet $S$ are all the elements in $S$ at distance $\|\rho_S\|$ of $\Ocal$.
\end{theo}

In general $\Ocal$ has a unique ancestor, $\Ocal$ itself.
But for example the orbit of $\rho^K$ has $2^r$ ancestors, where $r$ is the rank of $[K,K]$.

\section*{Acknowledgments}
We wish to thank
  the Research in Pairs program at
Mathematisches Forschungsinstitut Oberwolfach (February 2014), where our work on equivariant indices of Dirac operators was started.
The second author wish to thank Michel Duflo for his many comments.

\begin{center}
\bf Notations
\end{center}

Throughout the paper :
\begin{itemize}
\item $K$ denotes a compact connected Lie group with Lie algebra $\kgot$.
\item $T$ is a maximal torus in $K$ with Lie algebra $\tgot$.
\item $\Lambda\subset \tgot^*$ is the weight lattice  of $T$ : every $\mu\in \Lambda$ defines a $1$-dimensional
$T$-representation, denoted $\C_\mu$, where $t=\exp(X)$ acts by $t^\mu:= e^{i\langle\mu, X\rangle}$.

\item If $\hgot$ is a subalgebra of $\kgot$, we denote
by $(\hgot)$ its conjugacy class.

\item We fix a $K$-invariant Euclidean inner product $(\cdot,\cdot)$ on $\kgot$. This allows us to identify $\kgot$ and $\kgot^*$ when needed.

We denote by $\langle \cdot,\cdot\rangle$ the natural duality between $\kgot$ and $\kgot^*$.

\item We denote by $R(K)$ the representation ring of $K$ : an element $E\in R(K)$ can be represented
as finite sum $E=\sum_{\lambda\in\what{K}}\mm_\lambda V^K_\lambda$, with $\mm_\lambda\in\Z$. Here we have denoted by $V^K_\lambda$
the irreducible representation of $K$ indexed by $\lambda$.
\item We denote by $\hat R(K)$ the space of $\Z$-valued functions on $\hat K$. An element $E\in
    \hat R(K)$ can be represented
as an infinite sum $E=\sum_{\lambda\in\what{K}}\mm(\lambda) V^K_\lambda$, with $\mm(\lambda)\in\Z$.
An element of $\hat R(K)$ will be called a virtual representation of $K$.

\item If $H$ is a closed subgroup of $K$, the induction map $\mathrm{Ind}_H^K: \hat{R}(H)\to \hat{R}(K)$ is the dual of the restriction morphism $R(K)\to R(H)$. It is given by the Frobenius reciprocity formula:
     the multiplicity   $\mm(\lambda)$ of the irreducible representation $V_\lambda^K$ in
     $\mathrm{Ind}_H^K V_\mu^H$  is the multiplicity of $V_\mu^H$ in the restriction of the representation $V_\lambda^K$ to
     $H$.

\item If $E$ is a complex representation space of $H$, we denote simply by
$\bigwedge^{\bullet} E$ the  virtual representation of $H$ which is the difference of the representations of
$H$ in $\bigwedge^{\rm even} E$ and $\bigwedge^{\rm odd} E$.

\item The stabilizer of $\xi\in \kgot^*$  is denoted by $K_\xi$. It is a connected subgroup of $K$.
The Lie algebra of $K_\xi$ is denoted by $\kgot_\xi$. The element $\xi\in \kgot^*$ is called regular if $K_\xi$ is a Cartan subgroup of $K$.

\item In general, the notation $\Pcal$ is for any coadjoint orbit, while the notation $\Ocal$ is reserved for regular coadjoint orbits.

\end{itemize}

\section{Coadjoint orbits and Dixmier sheets}

Let $K$ be a compact connected Lie group with Lie algebra $\kgot$.

We first define the $\rho$-orbit.
Let $T$ be a Cartan subgroup of $K$.
Then $\tgot^*$ is imbedded in $\kgot^*$ as the subspace of $T$-invariant elements.
Roots will  be considered as elements of $\tgot^*$ (the corresponding character of $T$ being $e^{i\alpha}$).
Choose a system of positive roots $\Delta^+\subset {\mathfrak t}^*$, and let
$\rho^K=\frac{1}{2}\sum_{\alpha>0}\alpha$.
The definition of $\rho^K$ requires
  the choice of a Cartan subgroup $T$ and of a  positive root system. However a different choice leads to a conjugate element. Thus we can make the following definition.

\begin{defi}\label{def:s-O}
We denote by
$o(\kgot)$ the coadjoint orbit of $\rho^K\in \kgot^*$.
We call $o(\kgot)$ the $\rho$-orbit.
We denote by $\|\rho^K\|$ the norm of any point on $o(\kgot)$.
\end{defi}

\medskip

Let $\mathcal{H}_{\mathfrak k}$ be the set of conjugacy classes of the reductive algebras ${\kgot}_\xi,\xi\in\mathfrak k^*$.
The set $\mathcal{H}_{\mathfrak k}$  contains the conjugacy class $(\tgot)$ formed by the Cartan sub-algebras.
It contains also $(\kgot)$ (stabilizer of $0$).
\begin{rem}
If $\hgot=\kgot_\xi$, then $\hgot_\C$ is the Levi subalgebra of the parabolic subalgebra determined by $\xi$.
Parabolics are classified by subsets of simple roots. Thus there are $2^r$ conjugacy classes of parabolics if $r$ is the rank of $[\kgot,\kgot]$.
 However,  different conjugacy classes of parabolics  might give rise to the same conjugacy class of Levi subalgebras (as seen immediately for type $A_n$).
\end{rem}

We group the coadjoint orbits  according to the conjugacy class $(\hgot)\in \mathcal{H}_{\mathfrak k}$ of the stabilizer.
\begin{defi}
The Dixmier sheet $\kgot^*_{(\hgot)}$  is the set of orbits $K\xi$ with $\kgot_\xi$ conjugated to $\hgot$.
 \end{defi}
We also say that an element of $\kgot^*_{(\hgot)}$ is a coadjoint orbit of type $(\hgot)$.

If $(\hgot)=(\tgot)$, the corresponding Dixmier sheet is the set of regular coadjoint orbits.

If $(\hgot)=(\kgot)$, the corresponding Dixmier sheet is the set of $0$-dimensional orbits,
that is the orbits $\{\xi\}$ of the elements $\xi$
  of $\kgot^*$  vanishing on $[\kgot,\kgot] $.

\bigskip

We denote by $\mathcal S_{\mathfrak k}$ the set of conjugacy classes of the semi-simple parts $[\hgot,\hgot]$ of
the elements $(\hgot)\in \mathcal H_{\mathfrak k}$.
\begin{lem}\label{lem:conjugacy}
The map $(\hgot)\to( [\hgot,\hgot])$ induces a bijection
between $\mathcal H_{\mathfrak k}$ and $\mathcal S_{\mathfrak k}$.
\end{lem}
\begin{proof}
Assume that $[\hgot,\hgot]=[\hgot',\hgot']=\sgot$. Let $N$ be the normalizer of $\sgot$,
and  $\ngot$ its Lie algebra.
Thus  $\hgot$ and $\hgot'$ are  contained in $\ngot$. Let $\tgot,\tgot'$ be  Cartan subalgebras of $\hgot, \hgot'$.
Then $\tgot$ and $\tgot'$ are conjugated  by an element of $N$.
As $\hgot=\sgot+\tgot$, we see that $\hgot$ is conjugated to $\hgot'$. $\Box$
\end{proof}
\medskip


The connected Lie subgroup  with Lie algebra $\hgot$ is denoted $H$. Thus if $\hgot=\kgot_\xi$, then $H=K_\xi$.
We write $\hgot= \zgot \oplus [\hgot,\hgot]$ where $\zgot$ is the center and  $[\hgot,\hgot]$ is the semi-simple part of $\hgot$.
Thus $\hgot^*= \zgot^* \oplus [\hgot,\hgot]^*$ and
$\zgot^*$ is the set of elements in $ \hgot^*$ vanishing on the semi-simple part of $\hgot$.
We write $\kgot=\hgot\oplus [\zgot,\kgot]$, so we embed $\hgot^*$ in $\kgot^*$ as a $H$-invariant subspace,  that is
we consider an element $\xi\in \hgot^*$ also as an element of $\kgot^*$ vanishing on  $[\zgot,\kgot]$.

We  consider the $\rho$-orbit $o(\hgot)$ in $\hgot^*$. Remark that $o(\hgot)$  is contained in
$[\hgot,\hgot]^*$.
We denote by $\|\rho^H\|$ the norm of any point in $o(\hgot)$.

If $S$ is the Dixmier sheet associated to $(\hgot)$, we denote by
 $\|\rho^S\|$ the norm of  $\rho^H$.
%
%
%

\section{Admissible coadjoint orbits}

We will be interested in admissible coadjoint orbits.
This notion  is defined in  \cite{du} for any real Lie group.
 Let us now describe the set of admissible coadjoint orbits  in concrete terms when  $K$ is a compact connected Lie group.

Consider $\xi\in \kgot^*$.
We have  $\langle \xi, [\kgot_\xi,\kgot_\xi]\rangle=0$.
 If $i\theta  : \kgot_\xi\to i\R$
is the differential of a character of $K_\xi$, we denote by $\C_\theta$ the corresponding $1$-dimensional representation
of $K_\xi$, and by $[\C_\theta]=K\times_{K_\xi}\C_\theta$ the corresponding  line bundle over the coadjoint orbit $K\xi\subset\kgot^*$.

\begin{defi}\label{lem:integral}
An element $\xi\in\kgot^*$ is integral if  $i\xi$ is the differential of a $1$-dimensional representation of $K_\xi$. 
\end{defi}

This notion of integrability is invariant under the coadjoint action, so a coadjoint orbit is called integral when any of its element is integral.

The following notion of admissibility is necessary for defining  the spin quantization $\QS_K(\Pcal)$ of an orbit $\Pcal$.

 Let $K\xi$ be a coadjoint orbit. The quotient space $\kgot/\kgot_\xi$ is equipped with the symplectic
 form $\Omega_\xi(\bar{X},\bar{Y}):=\langle \xi,[X,Y]\rangle$, and with a unique
 $K_\xi$-invariant complex structure $J_\xi$ such that $\Omega_\xi(-,J_\xi -)$ is a scalar product.
We denote by $\qgot^{\xi}$ the space
$\kgot/\kgot_\xi$ {\bf considered as a complex vector space}  via the complex structure $J_\xi$.
Any element $X\in \kgot_\xi$ defines a complex linear map $\mathrm{ad}(X): \qgot^\xi\to \qgot^\xi$.
\begin{defi}\label{def-rho-a}
For any $\xi\in \kgot^*$, we denote by $\rho(\xi)$ the element of $\kgot_\xi^*$ such that
$$
\langle\rho(\xi),X\rangle= \frac{1}{2i}\tr_{\qgot^\xi}\mathrm{ad}(X), \quad X\in\kgot_\xi.
$$

\end{defi}

Note that $\rho(\xi)$ vanishes on $[\kgot_\xi,\kgot_\xi]$. 
As explained in the preceding section, we consider also $\rho(\xi)$ as an element of $\kgot^*$.

Notice that $2\rho(\xi)$ is integral since $2i\rho(\xi)$ is the differential of the character
$k\in K_\xi\mapsto \det_{\qgot^\xi}(k)$. 

Let $\tgot$ be a Cartan subalgebra and $\xi\in \tgot^*$. Then if  $\Delta$ is the root system with respect to $\tgot$, we have
$\rho(\xi)=\frac{1}{2}\sum_{\alpha\in \Delta, (\alpha,\xi)>0}\alpha.$

\bigskip

\begin{defi}\label{lem:admissible}
We say that
the coadjoint orbit  $K\xi$ is admissible if  $i(\xi-\rho(\xi))$ is the differential of a $1$-dimensional representation of $K_\xi$.
\end{defi}

 Remark that if the coadjoint orbit  $K\xi$ is admissible,
then   $2\xi$ is integral.

It is easy to see that, for a compact connected Lie group $K$, this definition of admissibility is equivalent to Duflo definition via the metaplectic correction.

In particular the orbit $o(\kgot)$ is admissible.
Indeed if $\xi=\rho^K$, then $\xi-\rho(\xi)=0$.

\begin{rem}
When the group $K$ is simply connected, then a regular admissible orbit  $K\xi$ is integral.
This is not the case when $K$ is not simply connected. For example the orbit $o(\kgot)$
(which is regular admissible) is
not integral when $K=SO(3)$.
However, even if the group $K$ is simply connected,
non regular admissible coadjoint orbits are not necessary integral.
See the example
\ref{exa:ExampleU3} below.
\end{rem}

\begin{defi}
We denote by $\Acal((\hgot))$ the set of admissible orbits of type
$(\hgot)$.
\end{defi}

We now define the shift of an orbit. If $\mu\in \kgot^*$, recall that $o(\kgot_\mu)\subset \kgot_\mu^*$ is the coadjoint orbit
of the $\rho$-element for the group $K_\mu$.
\begin{defi}\label{defi-shift}
 To any coadjoint orbit $\Pcal\subset\kgot^*$, we associate the coadjoint orbit $s(\Pcal)\subset \kgot^*$
which is defined as follows : if $\Pcal=K\mu$, take $s(\Pcal)=K\xi$ with $\xi\in \mu +o(\kgot_\mu)$.
We call $s(\Pcal)$ the shift of the orbit $\Pcal$.
\end{defi}
If $\Ocal$ is regular,  $s(\Ocal)=\Ocal$, and if $\Pcal=\{0\}$, then $s(\Pcal)=o(\kgot)$.

Beware that we have two elements of $\hgot^*=[\hgot,\hgot]^*\oplus \zgot^*$ associated to an element $\mu\in \kgot^*$ such that $\kgot_\mu=\hgot$: the element
$\rho(\mu)\in \zgot^*$ (defined canonically) and the element $\rho^{K_\mu}\in [\hgot,\hgot]^*$ (defined up to $H$-conjugacy).
Remark that $\rho(\mu)+\rho^{K_\mu}$ is conjugated to $\rho^K$.
More concretely,
if we choose a Cartan subgroup $T$ of $H$ with Lie algebra $\tgot$,
 and a positive root system $\Delta^+$ for roots of $\kgot$ with respect to $\tgot$,
 then, when $\mu$ is dominant,
$$\rho(\mu)=\frac{1}{2}\sum_{\alpha\in \Delta^+,(\alpha,\mu)>0}\alpha$$
and
$$\rho^{K_\mu}=\frac{1}{2}\sum_{\alpha\in \Delta^+,(\alpha,\mu)=0}\alpha.$$

The orbit $K\mu$ is admissible if $\mu-\rho(\mu)$ is in the weight lattice of $T$.
The shift of the orbit $K\mu$ is $K(\mu+\rho^{K_\mu})$.

\begin{exam}\label{exa:ExampleU3} Consider the group $K=SU(3)$.
 Then there are $3$ different sheets, determined (for this case) by the dimensions of orbits.
 The regular sheet consists of coadjoint orbits of  dimension $6$,
 The subregular sheet consists of coadjoint orbits of  dimension $4$,
finally we have the orbit $\{0\}$ of dimension $0$.

 We now parameterize the subregular sheet.
 Let $\hgot$ be the Lie algebra of $H=S(U(1)\times U(2))$.
Then the stabilizer $\kgot_f$ of a coadjoint orbit $Kf$ of dimension $4$ is conjugated to $\hgot$.
Let us give representatives of the orbits in $\kgot^*_{(\hgot)}$.

We consider the Cartan subalgebra of diagonal matrices and choose a Weyl chamber.
Let $\omega_1,\omega_2$ be the two fundamental weights.
Thus $\rho^K=\omega_1+\omega_2$.
Let $\sigma_1,\sigma_2$ be the half lines $\R_{>0}\omega_1$, $\R_{>0}\omega_2$.

Then  $\kgot^*_{(\hgot)}$ is the set of orbits
$K(t\omega_1)$, with $t\neq 0$.
Here we have privileged an element $f$ on the orbit $\Pcal\in \kgot^*_{(\hgot)}$ such that
$\kgot_f=\hgot.$
We could also privilege representatives of orbits  $\Pcal\in \kgot^*_{(\hgot)}$ belonging to the chosen closed Weyl chamber.
As $-\omega_1$  conjugated to $\omega_2$, we obtain the description of
$\kgot^*_{(\hgot)}$ as the set of orbits $\{K(t_1\omega_1), K(t_2\omega_2)\}$ with $t_1>0$ and $t_2>0$
(orbits of points in the boundary of the Weyl chamber, except $\{0\}$).

Let us now describe the set $\Acal((\hgot))$ of admissible coadjoint orbits of type $(\hgot)$.

We obtain, using  the first description,
that the set $\Acal((\hgot))$ is equal to the collection of orbits $K\cdot (\frac{1+2n}{2}\omega_1), n\in\Z$ 
(see Figure \ref{admissiblesSU2}).
Remark that  orbits in $\Acal((\hgot))$ are admissible, but not  integral.

\begin{figure}[!h]
\begin{center}
 \includegraphics[width=2.5 in]{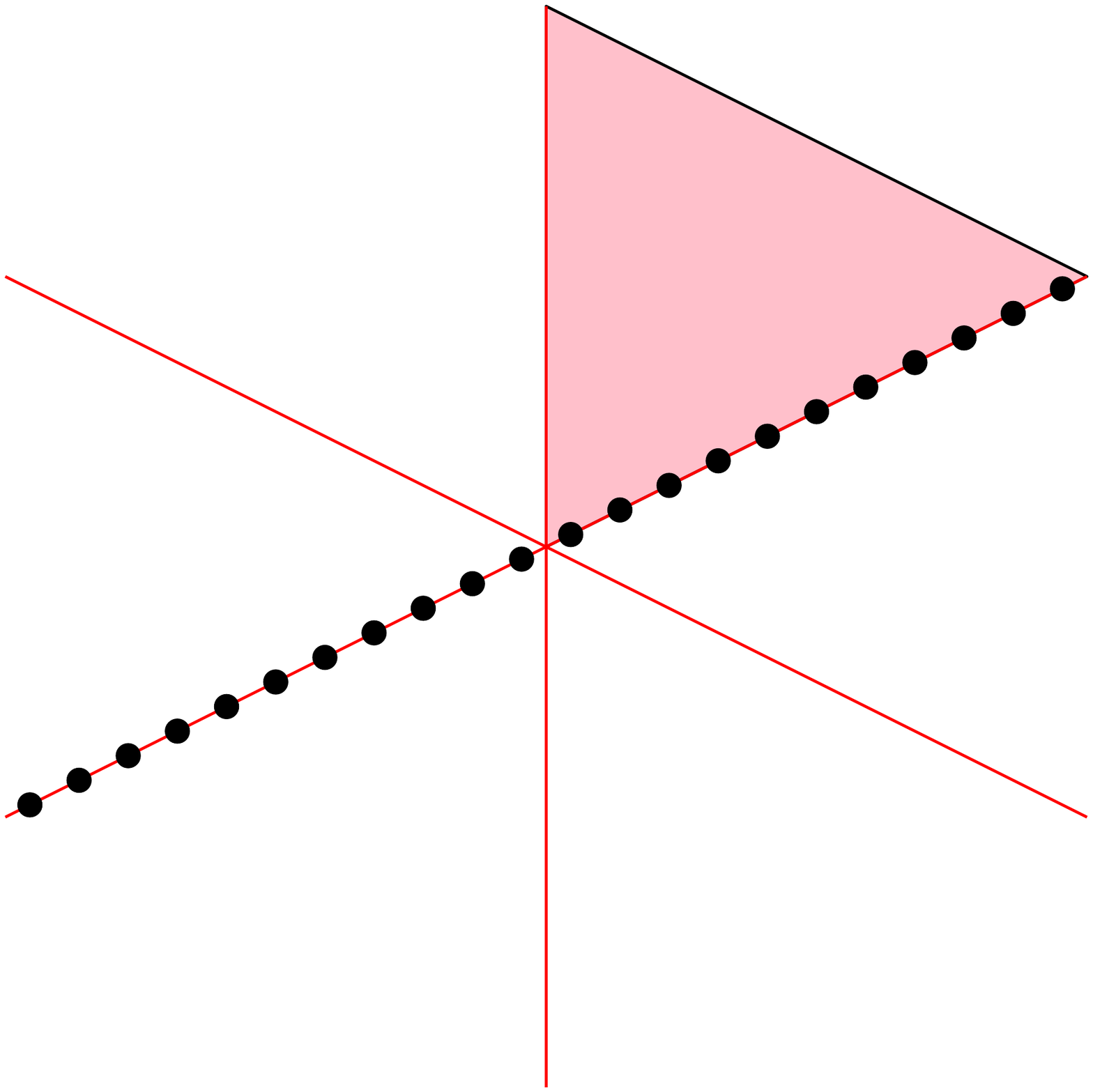}
\caption{ A set of representatives of admissible orbits of type $(\hgot)$}
 \label{admissiblesSU2}
\end{center}
\end{figure}

Using the second description, with representatives in the chosen closed Weyl chamber,
we  see that the set $\Acal((\hgot))$ is equal to the collection of orbits $K\cdot (\frac{1+2n}{2}\omega_i), n\in\Z_{\geq 0}, i=1,2$.

We have also depicted in this description the shifted orbits.
Remark that the shift of the admissible elements in $\Acal_{(\hgot)}$ are admissible, except for the two elements
 $\frac{1}{2}\omega_1$ and $\frac{1}{2}\omega_2$
 with shifts  $\omega_2$ and $\omega_1$ respectively, elements which are not admissible.

Finally remark that the orbit $K\rho^K$ is obtained as the shift of  the admissible orbits
$\{0\}$, $K(\frac{3}{2}\omega_1)$, $K(\frac{3}{2}\omega_2)$, and $K\rho_K$.

\begin{figure}[!h]
\begin{center}
  \includegraphics[width=2.5 in]{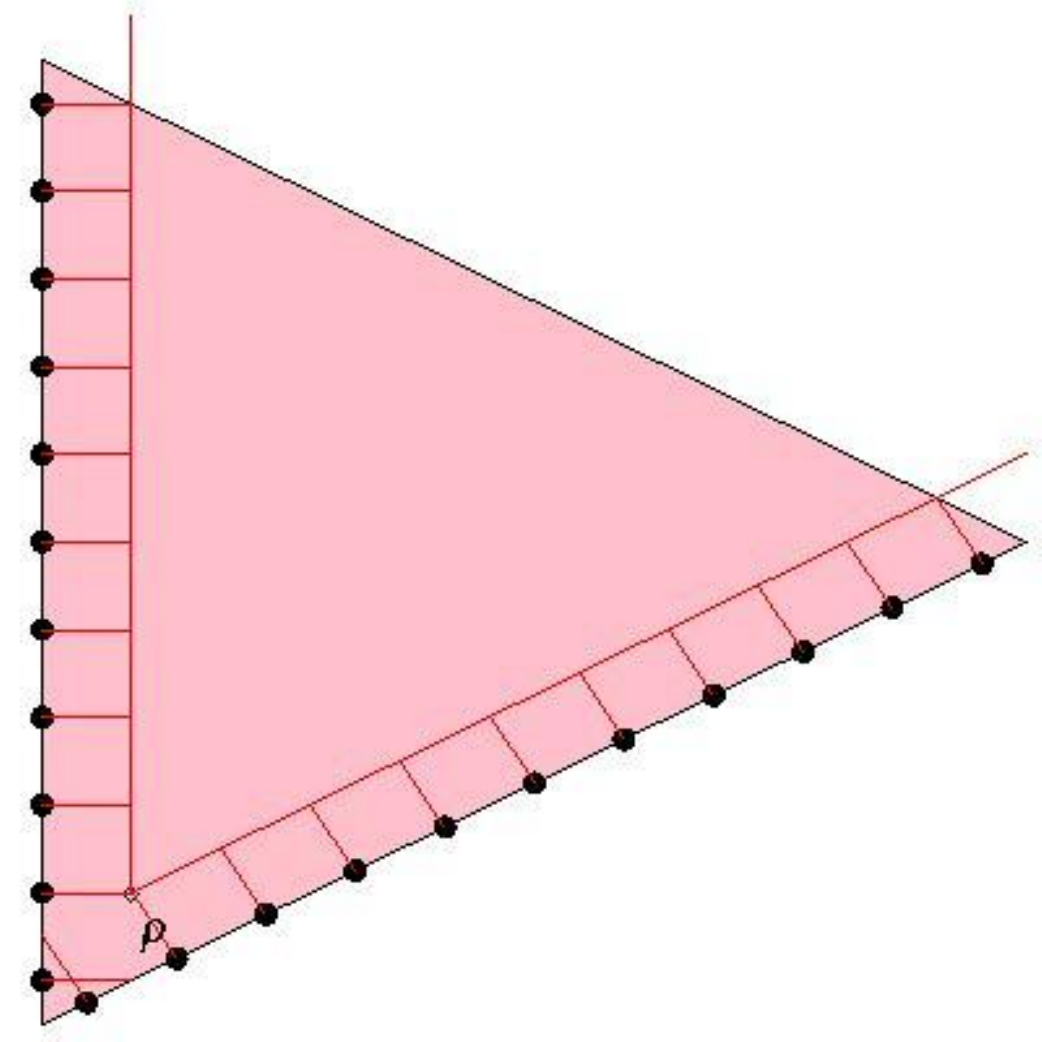}
\caption{ Another set of representatives of admissible orbits of type $(\hgot)$  and their shifts}
 \label{ancestorall}
\end{center}
\end{figure}
\end{exam}

\medskip

We now associate to an admissible coadjoint orbit an element of $R(K)$.

Let $\Pcal=K\xi$ be an admissible coadjoint orbit.
We consider the $\Z/2\Z$-complex space $\bigwedge^{\bullet} \qgot^\xi\otimes \C_{\xi-\rho(\xi)}=E^+\oplus E^-$ with $E^+=\bigwedge^{\rm even} \qgot^\xi\otimes \C_{\xi-\rho(\xi)}$
and
$E^-=\bigwedge^{\rm odd} \qgot^\xi\otimes \C_{\xi-\rho(\xi)}.$

We still denote by $\bigwedge^{\bullet} \qgot^\xi\otimes \C_{\xi-\rho(\xi)}$
the element $E^+-E^-$ in $R(H)$.

\begin{defi}\label{defi:QS}
We define $\QS_K(\Pcal)\in  R(K)$ by the formula:
\begin{equation}\label{eq:induction-index-Ocal}
\QS_K(\Pcal)=\mathrm{Ind}_{K_\xi}^K\left({\bigwedge}^{\bullet} \qgot^\xi\otimes \C_{\xi-\rho(\xi)}\right).
\end{equation}
\end{defi}
Thus $\QS_K(\Pcal)$ is the virtual representation $\mathrm{Ind}_{K_\xi}^K E^+-\mathrm{Ind}_{K_\xi}^K E^-$.
The fact that  $\QS_K(\Pcal)$ is a finite signed sum of representations of $K$  is  easy to see using Frobenius reciprocity formula.

We now interpret
$\QS_K(\Pcal)$ as the equivariant index of a twisted Dirac operator on $\Pcal$.
Consider the $K$-equivariant graded complex vector bundle $\Scal_\Pcal=K\times_{K_\xi} E$.
It defines a $\spinc$-bundle on $\Pcal$ (see \cite{pep-vergne:spin2}).
The determinant line bundle of $\Scal_\Pcal$ is the line bundle $[\C_{2 \xi}]$ and its corresponding moment map
(as defined in \cite{pep-vergne:spin2}) is the canonical injection $\Pcal\to \kgot^*$.
Consider the Dirac operator $D_\Pcal$ associated to  this $\spinc$-bundle.
Then by Atiyah-Bott \cite{AB}, the equivariant index of $D_\Pcal$ is  equal to $\mathrm{Ind}_{K_\xi}^K E^+-\mathrm{Ind}_{K_\xi}^K E^-$.
Thus  $\QS_K(\Pcal)$ coincides with the equivariant index of the  twisted Dirac operator $D_\Pcal$ (and  so belongs to $R(K)$).
For this reason $\QS_K(\Pcal)$ is called the spin quantization of $\Pcal$.

The following proposition is well known. We will  recall  its proof in Lemma \ref{lem:easyregular} in the next subsection.

\begin{prop}\label{prop:easyregular}
\begin{itemize}
\item The map $\Ocal\mapsto \pi_\Ocal:=\QS_K(\Ocal)$ defines a bijection between the set of regular admissible orbits and $\wK$.
\item  $\QS_K(o(\kgot))$ is the trivial representation of $K$.
\end{itemize}
\end{prop}

In Section \ref{sectionadmissibleandrep}, we will  describe the virtual  representation   $\QS_K(\Pcal)$ attached to any admissible orbit in terms of regular admissible orbits. It is either equal to $0$, or is an element of $\wK$.

\section{The magical inequality}\label{sec:admissible-Weyl}

In order to parameterize coadjoint orbits,
we choose a Cartan subgroup  $T$ of $K$ with Lie algebra $\tgot$.
Let $\Lambda \subset \tgot^*$ be the lattice of weights of $T$.
Let $W$ be the Weyl group.
 Choose a system of positive roots $\Delta^+\subset {\mathfrak t}^*$, and let as before
$$\rho^K=\frac{1}{2}\sum_{\alpha>0}\alpha.$$
If $\alpha\in \tgot^*$ is a root, we denote by $H_\alpha\in \tgot$ the corresponding coroot (so $\langle \alpha, H_\alpha\rangle=2$).
Then $\langle \rho^K,H_\alpha\rangle=1$ if and only if $\alpha$ is a simple root.

Define the positive closed Weyl chamber by
$$
\tgot^*_{\geq 0}=\{\xi\in \tgot^*;\, \langle \xi,H_\alpha\rangle\geq 0 \,\, {\rm for\, all}\ \alpha> 0\},
$$
and we denote by $\Lambda_{\geq 0}:=\Lambda\cap \tgot^*_{\geq 0}$ the set of dominant weights. Any coadjoint orbit
$\Pcal$ of $K$ is of the form $\Pcal= K\xi$ with $\{\xi\}=\Pcal\cap \tgot^*_{\geq 0}$.

We index the set $\wK$ of  classes of finite dimensional irreducible representations of $K$ by
the set $\rho^K +\Lambda_{\geq 0}$. The irreducible representation $\pi_\lambda$ corresponding to
$\lambda\in \rho^K +\Lambda_{\geq 0}$ is the irreducible representation with
infinitesimal character $\lambda$. Its highest weight is $\lambda-\rho^K$.
The representation $\pi_{\rho^K}$ is the trivial representation of $K$.
The Weyl character formula for the representation $\pi_\lambda$ is, for $X\in \tgot$,
$$
\tr \, \pi_\lambda(e^X)=\frac{\sum_{w\in W} \epsilon(w) e^{i\langle w\lambda,X\rangle}}{\prod_{\alpha>0}
e^{i\langle \alpha,X\rangle/2}- e^{-i\langle \alpha,X\rangle/2}}.
$$

For any $\mu\in \tgot^*$, recall its element $\rho(\mu)\in\kgot^*$ (Definition \ref{def-rho-a}).

 \begin{lem}\label{lem:easyregular}
 Let $\lambda\in \tgot^*_{\geq 0}$ be a regular admissible element of $\kgot^*$.
Then
\begin{enumerate}

\item $\rho(\lambda)=\rho^K$.
\item $\lambda\in \rho^K+\Lambda_{\geq 0}.$
\item$\QS_K(K\lambda)=\pi_\lambda.$
 \end{enumerate}
 \end{lem}

 \begin{proof}
 Let $\lambda\in\tgot^*_{\geq 0}$ be regular and admissible, then $\rho(\lambda)=\rho^K$, so $\lambda\in \{\rho^K+\Lambda\}\cap \tgot^*_{> 0}$.
 If $\alpha$ is a simple root, then
 {\bf the integer} $\langle \lambda-\rho^K,H_\alpha\rangle =\langle \mu,H_\alpha\rangle -1$ is non negative, as $\langle \lambda,H_\alpha\rangle >0$. So $\lambda-\rho^K$ is a dominant weight.

 Let $\Ocal=K\lambda$.
 Weyl character formula  coincide with Atiyah-Bott-Lefschetz formula \cite{AB}
 for the index of the Dirac operator $D_\Ocal$.
 Thus we obtain Lemma  \ref{lem:easyregular} and Proposition \ref{prop:easyregular}. $\Box$
\end{proof}

\medskip

\medskip
The positive Weyl chamber $\CW$ is the  cone determined by the  equations
$\langle \lambda, H_\alpha\rangle \geq 0$ for the simple roots  $\alpha\geq 0$.
We denote by $\mathcal{F}_\kgot$ the set of the relative interiors of the  faces of ${\mathfrak t}^*_{\geq 0}$.
Thus the set $\mathcal{F}_\kgot$ is parameterized by subsets of the simple roots and has $2^r$ elements, $r$ being the rank of $[\kgot,\kgot]$.
Thus ${\mathfrak t}^*_{\geq 0}=\coprod_{\sigma\in {\mathcal F_\kgot}} \sigma$, and we denote by
$ {\mathfrak t}^*_{>0}\in \mathcal{F}_\kgot$ the interior of ${\mathfrak t}^*_{\geq 0}$.

Let  $\sigma\in \mathcal{F}_\kgot$. Thus $\R\sigma$, the linear span of $\sigma$,  is
the subspace determined by $\langle \lambda, H_\alpha\rangle=0$ where the $\alpha$ varies over a subset of the simple roots.

The stabilizer $K_\xi$ does not depend of the choice of the point $\xi\in \sigma$ : we denote it by $K_\sigma$.
The map $\sigma\to \kgot_\sigma$ induces a surjective map  from
$\Fcal_\kgot$ to  $\Hcal_\kgot$. This map may not be injective: in Example \ref{exa:ExampleU3}, $\sigma_1$ and $\sigma_2$ leads to the
same element of $\Hcal_\kgot$, as the corresponding groups $K_{\sigma_1}=S(U(1)\times U(2))$ and  $K_{\sigma_2}=S(U(2)\times U(1))$ are conjugated.

For $\sigma\in \Fcal_\kgot$, we have the decomposition
$\kgot_\sigma=[\kgot_\sigma,\kgot_\sigma]\oplus \zgot(\kgot_\sigma)$ with dual decomposition
$\kgot_\sigma^*=[\kgot_\sigma,\kgot_\sigma]^*\oplus \R\sigma$.
Let
$$
\rho^{K_\sigma}:=\frac{1}{2}\sum_{\stackrel{\alpha>0}{(\alpha,\sigma)=0}}\alpha
$$
be the $\rho$-element of the group $K_\sigma$  associated to the positive root system
$\{\alpha>0,(\alpha,\sigma)=0\}$ for $K_\sigma$.
Then $$\rho^K-\rho^{K_\sigma}=\frac{1}{2}\sum_{\stackrel{\alpha>0}{(\alpha,\sigma)>0}}\alpha,$$
and for any $\mu\in \sigma$,  the element $\rho(\mu)\in \kgot^*$ is equal to
$\rho^K-\rho^{K_\sigma}$.
In particular, $\rho^K-\rho^{K_\sigma}$ vanishes on $[\kgot_\sigma, \kgot_\sigma]$, so
  $\rho^K-\rho^{K_\sigma}\in \R \sigma$, while $\rho^{K_\sigma}\in [\kgot_\sigma,\kgot_\sigma]^*$.
The decomposition $\rho^K=(\rho^K-\rho^{K_\sigma})+\rho^{K_\sigma}$ is an orthogonal decomposition.

Figure \ref{figrho} shows this orthogonal decomposition of $\rho$ for the case $SU(3)$.

\begin{figure}[!h]
\begin{center}
  \includegraphics[width=2 in]{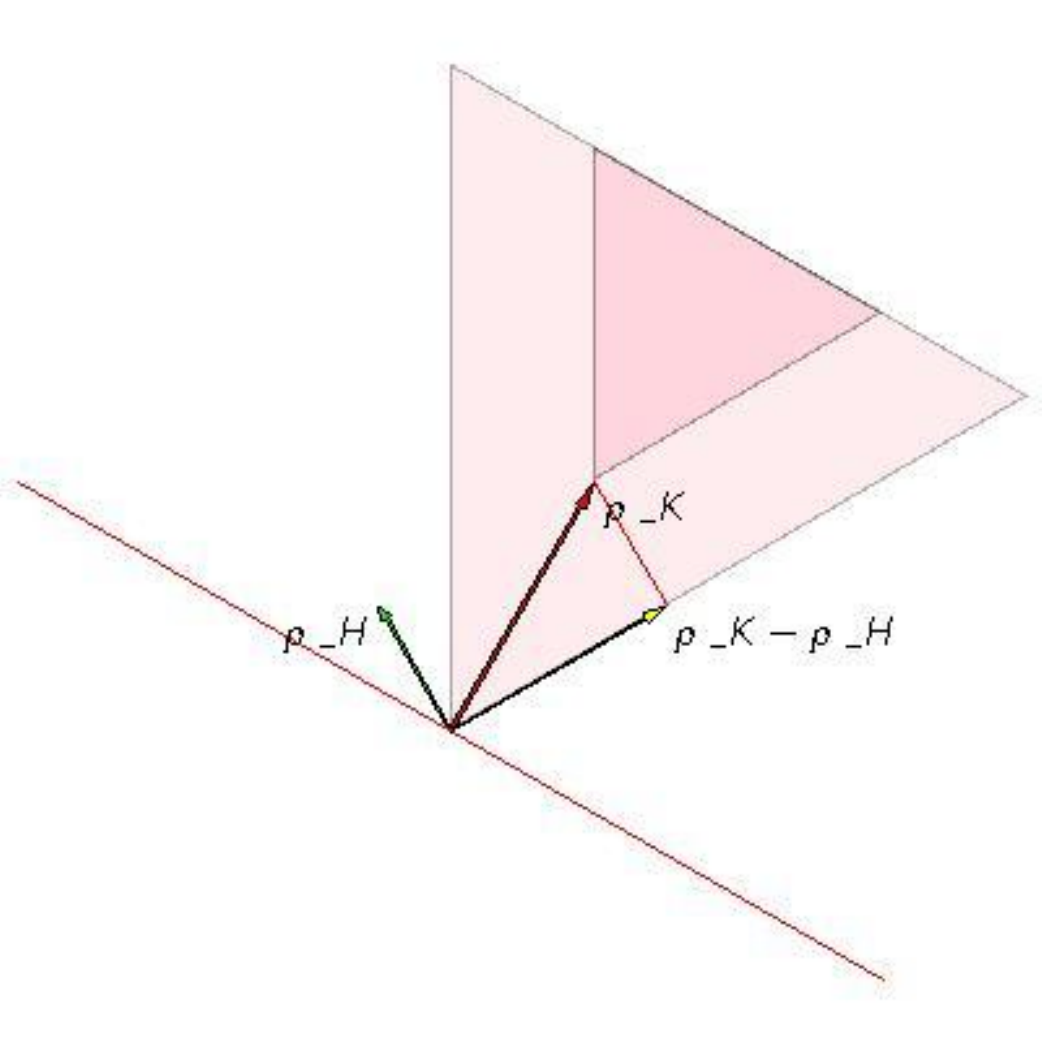}
\caption{Orthogonal decomposition of  $\rho_K$}
 \label{figrho}
\end{center}
\end{figure}

%

The subset $\rho_K+\CW$ of the positive Weyl chamber will be called the shifted Weyl chamber. It is determined by the inequalities $\langle \lambda, H_\alpha\rangle \geq 1$ for any simple root  $\alpha\geq 0$, and thus
$\langle \lambda, H_\alpha\rangle \geq 1$ for any positive root.
Let us list some  immediate  properties of the shifted Weyl chamber.

\begin{prop}\label{prop:beforeinfernal}
\begin{enumerate}
\item

 If $\lambda\in\rho^K+\CW$, then $( \lambda,\lambda) \geq (\lambda,\rho^K)\geq (\rho^K,\rho^K)$.
The equality $( \lambda,\lambda)=(\lambda,\rho^K)$ holds only if $\lambda=\rho^K$.

\item
Let $\sigma\in \Fcal_{\kgot}$.

$\bullet$
The orthogonal projection of $\xi\in\tgot^*_{> 0}$ onto $\R\sigma$ belongs to $\sigma$.

$\bullet$  We have $\rho^K-\rho^{K_\sigma}\in \sigma$ for any $\sigma\in\Fcal_\kgot$.

\item For any $(\hgot)\in \Hcal_\kgot$,
$\|\rho^K\| \geq \|\rho^H\|$, and $\|\rho^K\|= \|\rho^H\|$ only if $H=K$.

\end{enumerate}

\end{prop}

\begin{proof}
If $\lambda=\rho^K+c$, with $c\in \CW$,  inequalities
$( \lambda,\lambda)\geq  (\lambda,\rho^K)\geq (\rho^K,\rho^K)$ follows from the fact that $(\lambda,c)$ and $(\rho^K,c)$ are non negative,
as the scalar product of two elements of $\CW$ is non negative. Equality $( \lambda,\lambda)=  (\lambda,\rho^K)$ implies $\lambda=\rho^K$.

The second point   follows from the fact that the
 dual cone to $\CW$ is generated by the simple roots $\alpha_i$, and
$(\alpha_i,\alpha_j)\leq 0$, if $i\neq j$.

We have the orthogonal decomposition $\rho^K=\rho^{K_\sigma}+ (\rho^K-\rho^{K_\sigma})$: hence
$\rho^K-\rho^{K_\sigma}$, which is the orthogonal projection of $\rho^K$ on $\R\sigma$, belongs to $\sigma$.

For the third point, we might choose $H$ conjugated to $K_\sigma$, so  $\|\rho^K\|^2=\|\rho^{K_\sigma}\|^2+
\|\rho^K-\rho^{K_\sigma}\|^2$. $\Box$
\end{proof}

\medskip

The following theorem analyze the distance of a point in $\tgot^*$ to the shifted Weyl chamber. It is illustrated in Figure \ref{distance} in the case $SU(3)$.

\begin{figure}[!h]
\begin{center}
  \includegraphics[width=2 in]{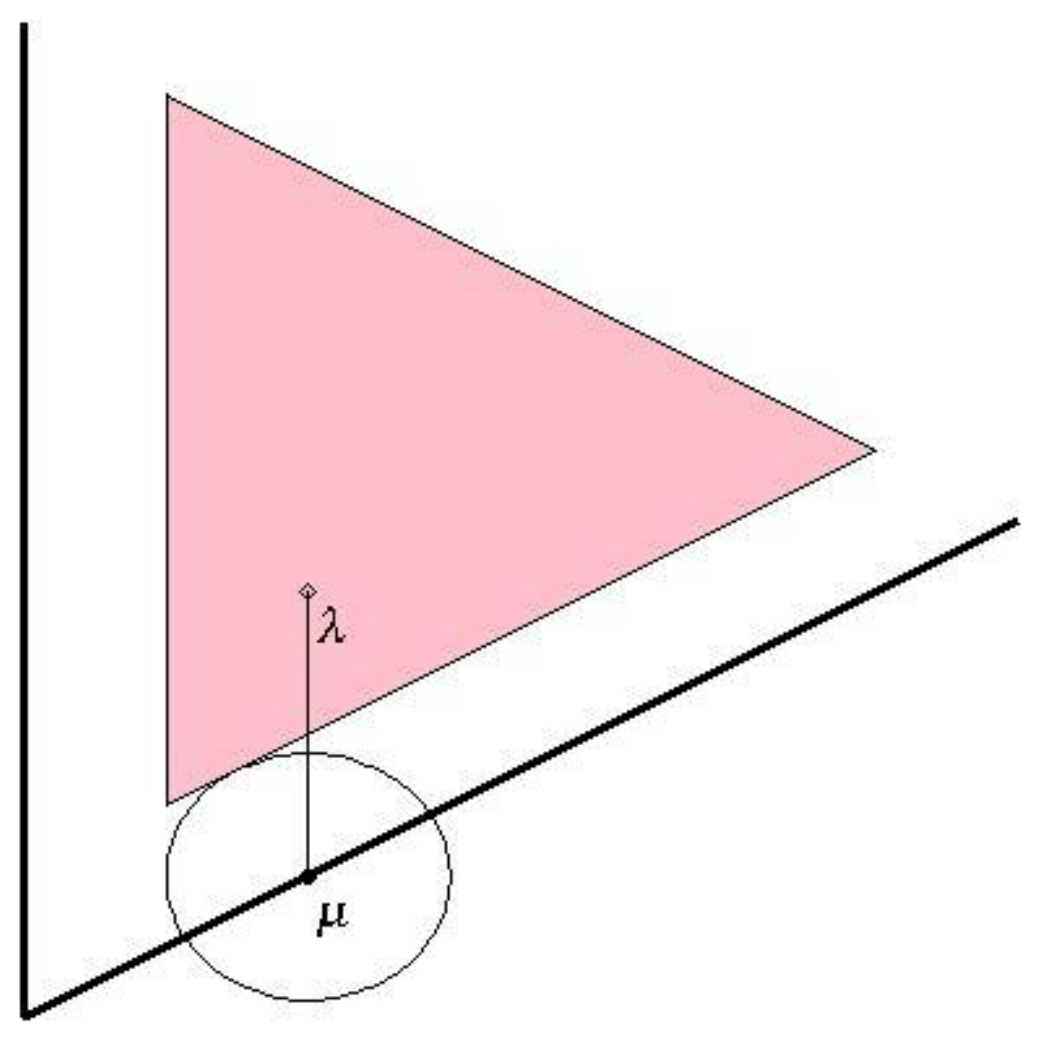}
\caption{ Distance of a singular element $\mu$  to an element
$\lambda$ in the shifted Weyl chamber}
 \label{distance}
\end{center}
\end{figure}

\begin{theo} \label{prop:infernal}

$\bullet$
 If $\lambda\in\rho^K+\CW$ and $\mu\in\tgot^*$, then:
 \begin{equation} \label{eq:crucial}
 \|\lambda-\mu\|^2\geq
\frac{1}{2}
\sum_{\stackrel{\alpha>0}{(\alpha,\mu)=0}}
(\lambda,\alpha).
\end{equation}

$\bullet$
If $\lambda\in\rho^K+\CW$ and $\mu\in\tgot^*$, then:
\begin{equation}\label{eq:crucial2}
\frac{1}{2}
\sum_{\stackrel{\alpha>0}{(\alpha,\mu)=0}}
(\lambda,\alpha)
\geq \|\rho^{K_\mu}\|^2.
\end{equation}

$\bullet$
If one of the inequalities  (\ref{eq:crucial}) or (\ref{eq:crucial2}) is an equality,
 then
 $\mu$ belongs to $\CW$, and  $\lambda=\mu+\rho^{K_\sigma}$ where $\sigma\in \Fcal_\kgot$ is the stratum of $\CW$ containing $\mu$.
Thus the two inequalities (\ref{eq:crucial}) or (\ref{eq:crucial2}) are equalities.
 In particular, $\lambda-\rho(\lambda)=\mu-\rho(\mu)$.
\end{theo}

\begin{proof}
Let $\kgot_\mu$ be the centralizer of $\mu$ and let $\zgot$ be the center of $\kgot_\mu$.
 Consider the orthogonal decomposition $\tgot^*=\zgot^*\oplus \agot^*$ where $\agot$ is a Cartan subalgebra for $[\kgot_\mu,\kgot_\mu]$, that is $\agot=\sum_{(\alpha,\mu)=0} \R H_\alpha$.
Let $\rho^{K_\mu}\in \agot^*$ be the $\rho$ element
for the system
$\Delta_+^1=\{\alpha>0, (\alpha,\mu)=0\}$ of $[\kgot_\mu,\kgot_\mu]$.

Let us write $\lambda=\rho^K+c$, with $c$ dominant,
 and
decompose $\rho^K=p_0+p_1$, $c=c_0+c_1$, with
$p_0,c_0\in \zgot^*$, $p_1,c_1\in \agot^*$.
Thus $\lambda=\lambda_0+\lambda_1$, with
$\lambda_0\in \zgot^*$ and $\lambda_1=p_1+c_1$.
Now $p_1$ belongs to
 the shifted  Weyl chamber in $\agot^*$.
 Indeed, for any $\alpha>0$ such that $(\alpha,\mu)=0$, we have
$\langle p_1,H_\alpha\rangle=\langle \rho^K,H_\alpha\rangle\geq 1.$
Similarly $c_1$ is dominant
for the system $\Delta_+^1$.

As $\mu\in \zgot^*$, we have
$\|\lambda-\mu\|^2= \|\lambda_0-\mu\|^2+\|p_1+c_1\|^2.$
Using the first point of \ref{prop:beforeinfernal}, we obtain
$$\|\lambda-\mu\|^2=\|\lambda_0-\mu\|^2+\|p_1+c_1\|^2\geq (p_1+c_1,  \rho^{K_\mu})\geq  \|\rho^{K_\mu}\|^2.$$

As  $$(p_1+c_1,  \rho^{K_\mu})=(\lambda,\rho^{K_\mu})= \frac{1}{2}
\sum_{\stackrel{\alpha>0}{(\alpha,\mu)=0}}
(\lambda,\alpha)$$ we obtain  Inequalities (\ref{eq:crucial}) and  (\ref{eq:crucial2}).

If the inequality  $\|\lambda-\mu\|^2\geq (p_1+c_1,  \rho^{K_\mu})$ is an equality, then
$$\|\lambda_0-\mu\|^2+\left(\|p_1+c_1\|^2-(p_1+c_1,  \rho^{K_\mu})\right)=0.$$
Both terms of the left hand side are non negative.
 Thus, using again the first point of Proposition \ref{prop:beforeinfernal}, we see that necessarily
$c_1=0$, $p_1=\rho^{K_\mu}$, and $\lambda_0=\mu$.
Thus for roots $\alpha\in \Delta_+^1$,
$\langle \rho^{K_\mu},H_\alpha\rangle=\langle \rho^K,H_\alpha\rangle$.
As $\rho^{K_\mu}$ takes value $1$ on simple roots for $K_\mu$, it follows that the set $S_1$ of simple roots for the system
 $\Delta_+^1$ is contained in the set of simple roots for $\Delta^+$.
 As $\agot=\oplus_{\alpha\in S_1}\R H_\alpha$, the orthogonal $\zgot$ of $\agot$ is  $\R \sigma$ for the face $\sigma$ of $\tgot^*$ orthogonal to the subset  $ S_1$ of simple roots.
We then have  $K_\mu=K_\sigma$.
Furthermore,
$\lambda=\mu+\rho^{K_\sigma}$.
Thus $\mu$ is the projection of $\lambda$ on $\R \sigma$, so $\mu\in \sigma\subset \CW.$ As $\rho(\lambda)=\rho^K$, and $\rho(\mu)=\rho^K-\rho^{K_\sigma}$, we obtain $\lambda-\rho(\lambda)=\mu-\rho(\mu)$.
So all assertions are proved. $\Box$
\end{proof}

\medskip

In this article, we will only use the following obvious corollary of Inequalities (\ref{eq:crucial})
and (\ref{eq:crucial2}). However, in our application \cite{pep-vergne:cras}, we will need both inequalities.
\begin{coro} (The magical inequality) \label{cor:magic}
 If $\lambda\in\rho^K+\CW$ and $\mu\in\tgot^*$, then:
 \begin{equation} \label{eq:magic}
 \|\lambda-\mu\|\
\geq \|\rho^{K_\mu}\|.
\end{equation}
The equality holds if and only if
$\mu$ is dominant and $\lambda-\rho(\lambda)=\mu-\rho(\mu)$.
\end{coro}

Let us give some consequences of the magical inequality (\ref{eq:magic}).
Let us define the notion of very regular element.

 \begin{defi}
 Let $\lambda\in \kgot^*$  be a regular element.
 Then $\lambda$ determines a closed positive Weyl chamber $C_\lambda\subset \kgot_\lambda^*$.
We say that $\lambda$ is very regular if
$\lambda\in \rho(\lambda)+C_\lambda$.

\end{defi}

In other words, a very regular element is an element which is conjugated to an element of the shifted Weyl chamber. Note that regular admissible elements are very regular.

\begin{coro}
Let $\lambda,\mu$ be two elements of $\kgot^*$.
Assume that $\lambda$ is very regular, then $\|\lambda-\mu\|\geq \|\rho^{K_\mu}\|$. The equality holds if and only if
$\mu\in C_\lambda$ and $\lambda-\rho(\lambda)=\mu-\rho(\mu)$.
\end{coro}

\begin{proof}
Consider the minimum  of
$\|\lambda-\mu'\|^2$ when $\mu'$ varies in $K\mu$.
Using the differential, we see that a point where the minimum is reached  is in the Cartan subalgebra determined by $\lambda$.
We can then conclude by  Corollary (\ref{cor:magic}). $\Box$
\end{proof}

\begin{figure}[!h]
\begin{center}
  \includegraphics[width=2 in]{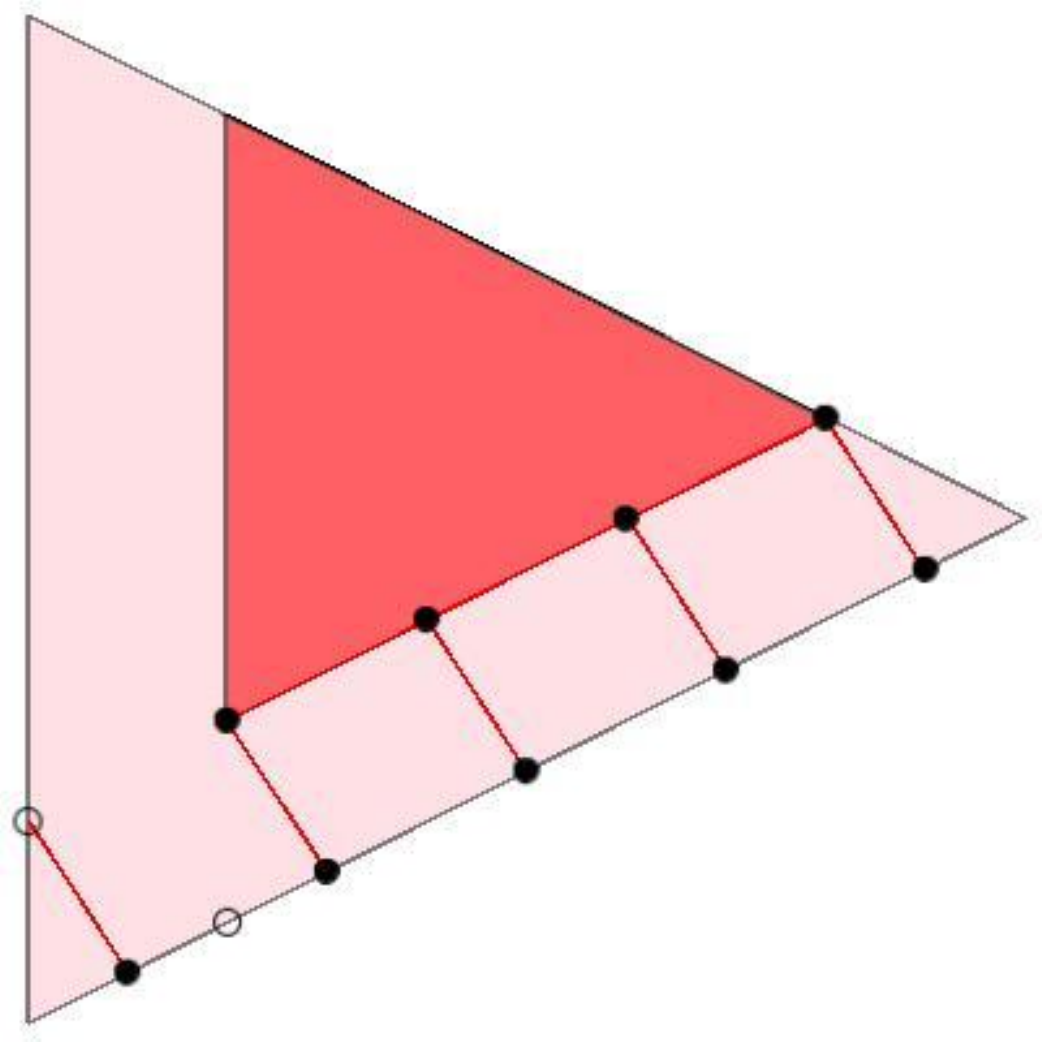}
\caption{Shifts  of admissible orbits}
 \label{fig:shifts}
\end{center}
\end{figure}

As consequence, we obtain Theorem \ref{theoinfernalintro} stated in the introduction.
 \begin{coro}\label{corodistance}
 The distance between a regular admissible coadjoint orbit $\Ocal$ and an orbit
$\Pcal$ in a Dixmier sheet $S$ is greater or equal to $\|\rho_S\|$.
\end{coro}

Let us now study the admissible coadjoint orbits  and their shifts.
We see that if $\mu$ is in $\sigma$, its shift is  $s(K\mu)=K(\mu+\rho^{K_\sigma})$.
Furthermore, $\mu\in \sigma$ is admissible if and only if $\mu+\rho^{K_\sigma}\in \rho^K+\Lambda$.

Theorem \ref{prop-general-admissible} below says  in particular that  if a shift
$\mu \to \mu+ \rho^{K_\sigma}$ of  an admissible element $\mu\in \sigma$ is regular, then it is in the shifted  Weyl chamber $\rho_K+\CW$.

Figure \ref{fig:shifts} illustrate this fact in the case $SU(3)$.
In this figure, the two points $\omega_1,\omega_2$ (which are not admissible) are represented by
the empty circles.
The shift of the admissible element $\frac{1}{2}\omega_1$ is $\omega_2$, and so is singular and is not admissible.

 \begin{theo}\label{prop-general-admissible}
Let $\sigma$ be a relative interior of a face of $\CW$.

\begin{enumerate}

\item
Let $\mu\in \sigma$.  If  $\mu+\rho^{K_\sigma}$  is regular admissible, then
$\mu$ is admissible.

\item
If $\mu\in\sigma$ is admissible, and  $\mu+\rho^{K_\sigma}$ is regular, then $\mu+\rho^{K_\sigma}$  is regular  admissible.
Moreover $\mu + \rho^{K_\sigma}\in \rho^K+ (\Lambda_{\geq 0}\cap \overline \sigma)$.
In particular $\mu + \rho^{K_\sigma}$ is in the shifted Weyl chamber.

\item  If $\mu\in\sigma$ is admissible, we have
$$
\QS_K(K\mu)=
\begin{cases}
   0\qquad\qquad {\rm if}\ \mu+\rho^{K_\sigma}\ \mathrm{is\ singular},\\
   \pi_{\mu+\rho^{K_\sigma}}\qquad \, {\rm if}\ \mu+\rho^{K_\sigma}\ \mathrm{is\ regular}.
\end{cases}
$$
\end{enumerate}
\end{theo}

\begin{proof}
Let us prove the two first points.

Let $\mu\in \sigma$ such that
$\lambda=\mu+\rho^{K_\sigma}$ is regular
admissible. Thus
$\|\lambda-\mu\|^2=\|\rho^{K_\mu}\|^2$.
The point $\lambda$ being regular and admissible, $\lambda$ is very regular.
 We use the magical inequality.
The equality $\|\lambda-\mu\|^2=\|\rho^{K_\mu}\|^2$ implies  $\lambda-\rho(\lambda)=\mu-\rho(\mu)=\mu- (\rho^K-\rho^{K_\sigma})$.
Thus $\rho(\lambda)=\rho^K$, so $\lambda\in \CW$. Furthermore $\mu$ is admissible.

Conversely, we see that if $\mu$ is admissible, and $\lambda=\mu+\rho^{K_\sigma}$ is regular, then $\lambda$ is regular admissible.
 We have seen that $\lambda\in \CW$.
 As $\lambda$ is regular admissible, it is in $\rho^K+\CW$.
The element $\lambda-\rho^K=\mu-(\rho^{K}-\rho^{K_\sigma})$ is in $\R \sigma$. As it is dominant, it is in $\overline{\sigma}$.

Let us prove the last point. Let $\qgot^{\mu}$ be the complex space   $\kgot/\kgot_{\mu}$ equipped with the complex structure $J_\mu$.
The equivariant index $\Theta$ of the Dirac operator $D_{K\mu}$ associated to  the $\spinc$-bundle
$\Scal_{K\mu}=K\times_{K_\mu}(\bigwedge^{\bullet} \qgot^\mu\otimes \C_{\mu-\rho^K+\rho^{K_\sigma}})$ is given by Atiyah-Bott fixed point formula: for $X\in \tgot$,
$\Theta(e^X)=\sum_{w\in W/W_\mu} w\cdot \frac{e^{i\langle \mu,X\rangle}}{\prod_{ \langle \alpha,\mu\rangle>0}e^{i\langle \alpha,X\rangle/2}-e^{-i\langle \alpha,X\rangle/2}}$.
Here $W_\mu$, the stabilizer of $\mu$ in $W$,
is equal to the Weyl group of the group $K_\sigma$.
Using $\sum_{w\in W_\sigma}\epsilon(w)e^{w\rho^{K_\sigma}}=
\prod_{\alpha>0, \langle \alpha,\sigma\rangle=0}(e^{\alpha/2}-e^{-\alpha/2})$, we obtain
\begin{equation}\label{equ:induced}
\Theta(e^X)= \frac{
\sum_{w\in W} \epsilon(w)e^{i\langle w(\mu+\rho^{K_\sigma}),X\rangle}}{\prod_{ \alpha>0}e^{i\langle \alpha,X\rangle/2}-e^{-i\langle \alpha,X\rangle/2}}.
\end{equation}

If $\mu+\rho^{K_\sigma}$ is singular, $\Theta$  is equal to zero.
If  $\mu+\rho^{K_\sigma}$ is regular, thanks to the first point,  $\mu+\rho^{K_\sigma}$
is in $\rho^K+\Lambda_{\geq 0}$, so
$\Theta=\pi_{\mu+\rho^{K_\sigma}}$. $\Box$

\end{proof}

\bigskip

We  proved that if  $\mu\in \sigma$ is admissible {\bf and}
 its shift $\mu+\rho^{K_\sigma}$ is regular, then  $\mu+\rho^{K_\sigma}$ is admissible and dominant.
 However, as illustrated in the following examples, if $\mu+\rho^{K_\sigma}$ is singular,

  $\bullet$
    $\mu+\rho^{K_\sigma}$
  is not necessarily dominant,

$\bullet$
 $\mu+\rho^{K_\sigma}$
might be admissible or not admissible.

\begin{exam}
Consider the group $K=U(7)$ and the face \break  $\sigma=\{\lambda_1>\lambda_2=\lambda_3=\lambda_4=\lambda_5=\lambda_6>\lambda_7\}$ of the Weyl chamber.
Here $\rho^K= (3,2,1,0,-1,-2,-3)$ and $\rho^{K_\sigma}= (0,2,1,0,-1,-2,0)$. Then $\mu:=(1,0,0,0,0,0,-1)$ is an admissible element of $\sigma$, but its shift $\mu+\rho^{K_\sigma}=(1,2,1,0,-1,-2,0)$ is singular, not admissible and  not dominant.

\end{exam}

\begin{exam}
Consider the group $K=U(5)$ and the face $\sigma=\{\lambda_1=\lambda_2>\lambda_3>\lambda_4=\lambda_5\}$ of the Weyl chamber.
Then
$\mu=(1/2,1/2,0,-1/2,-1/2)$ is an admissible element of $\sigma$.
  Its shift $\mu+\rho^{K_\sigma}=(1,0,0,0,-1)$  is singular and admissible.
\end{exam}

For application to the equivariant index of Dirac operators on general $\spinc$ $K$-manifolds,
 we  reformulate   Inequality (\ref{eq:crucial}) independently of a choice of a Weyl chamber using normalized traces.

\begin{defi}
Let $N$ be a real vector space and $b: N\to N$ a linear transformation, such that $-b^2$ is diagonalizable
with non negative eigenvalues, and  let $|b|=\sqrt{-b^2}$.
We denote by $\ntr_N|b|=\frac{1}{2}\tr_N|b|$, that is half of the trace of the action of $|b|$ in the real vector space $N$.
We call  $\ntr_N|b|$  the normalized trace of $b$.

\end{defi}

If  $N$ is an Euclidean space and $b$ a skew-symmetric  transformation of $N$,
then $-b^2$ is diagonalizable with  non negative eigenvalues.

For any $b\in\kgot$ and $\mu\in \kgot^*$ fixed by $b$,
we may consider the action $\ad(b) : \kgot_\mu\to \kgot_\mu$. The corresponding normalized trace
$\ntr_{\kgot_\mu}|\ad(b)|$ is denoted simply by
$\ntr_{\kgot_\mu}|b|.$

\begin{prop}\label{prop:infernal-with-trace}
Let $b\in \kgot$ and
 denote by $\beta$ the corresponding element in $\kgot^*$.
Let $\lambda,\mu$ be elements of $\kgot^*$ fixed by $b$.
 Assume that $\lambda$ is very regular and that $\mu-\lambda=\beta$.
Then
 $$\|\beta\|^2 \geq \frac{1}{2} \ntr_{\kgot_\mu} |b|.$$
If the equality holds, then
 $\mu$ belongs to  the positive Weyl chamber  $C_\lambda$ and
\begin{enumerate}
\item $\lambda-\rho(\lambda)= \mu- \rho(\mu)$, hence $\lambda$ is admissible if and only if $\mu$ is admissible,
\item $s(K\mu)=K\lambda$.
\end{enumerate}
\end{prop}

\begin{proof}
Indeed, as  $\lambda$ is  fixed by $b$, we see that $\beta$ belong to  $\kgot_\lambda^*$. We may assume that
$\kgot_\lambda^*=\tgot^*$.
 Thus $\beta,\lambda$ and $\mu=\lambda-\beta$
belong to $\tgot^*$.
The element $\lambda$ is a very regular element of $\tgot^*$. Note that the element of $\kgot$ corresponding to $\mu$   acts trivially on $\kgot_\mu$. So  the inequality of Proposition \ref{prop:infernal-with-trace} is a restatement  of Inequality \ref{eq:crucial}.
If the equality holds, we apply Theorem \ref{prop:infernal} and we obtain the proposition. $\Box$
\end{proof}

\section{Admissible coadjoint orbits and associated representations}\label{sectionadmissibleandrep}

In this section, we give some more information on the map $\Pcal\to \QS_K(\Pcal)$.
The following proposition follows almost immediately from Theorem \ref{prop-general-admissible}.

\begin{theo}\label{theo:notregular}
 Let $\Pcal$ be an admissible orbit.
\begin{itemize}

\item  $\Pcal^*:=-\Pcal$ is also admissible and $\QS_K(\Pcal^*)=\QS_K(\Pcal)^*.$

\item If $s(\Pcal)$  is not regular, then $\QS_K(\Pcal)=0$.

\item If  $s(\Pcal)$  is regular, then $s(\Pcal)$ is a regular  admissible  coadjoint orbit and
 $\QS_K(\Pcal)=\QS_K(s(\Pcal))=\pi_{s(\Pcal)}$.
\end{itemize}
\end{theo}

For the remaining part of this section, we fix a conjugacy class $(\hgot)$.
Let $\kgot^*_{(\hgot)}$ be the Dixmier sheet determined by $(\hgot)$.

\begin{defi}\label{defi:h-ancestor}
Let $\Ocal\subset \kgot^*$ be a regular  orbit. A $K$-orbit $\Pcal$  in
 is called a $(\hgot)$-ancestor of $\Ocal$ if  $\Pcal\in \kgot^*_{(\hgot)}$ and $s(\Pcal)=\Ocal$.
\end{defi}

\begin{lem}
If $\Pcal\in \kgot^*_{(\hgot)}$ and $s(\Pcal)$ is regular, then $s(\Pcal)$ is admissible if and only
$\Pcal$ is admissible.
\end{lem}
 \begin{proof}
 Let $\Ocal=s(\Pcal)$. Assume that $\Ocal$ is admissible.
We may assume that $\Ocal=K\lambda$ with $\lambda\in \rho^K+\CW$ regular admissible, and
$\Pcal=K\mu$, with $\mu\in \CW$. Let $\sigma$ be the stratum of $\CW$ containing $\mu$.
By Theorem \ref{prop:infernal-with-trace},  $\lambda=\mu+\rho^{K_\sigma}$. We have $\mu-\rho(\mu)=\lambda-\rho^K$, so $\mu$ is admissible.
The converse is proved the same way. $\Box$
\end{proof}

\begin{theo}\label{ancestorsanddistances}

 Let $\Ocal\subset \kgot^*$ be a regular admissible orbit.
\begin{itemize}
\item
 If $\Pcal$ is a  $(\hgot)$-ancestor of $\Ocal$, then $\Pcal$ is admissible and at distance
 $\|\rho^H\|$ of $\Ocal$.

 \item If $\Pcal$ is an element in $\kgot^*_{(\hgot)}$ at distance $\|\rho^H\|$ of $\Ocal$, then $\Pcal$ is admissible,
  $s(\Pcal)=\Ocal$ and $\QS_K(\Pcal)=\QS_K(\Ocal)$.
\end{itemize}

\end{theo}

\begin{proof}
We need only to prove the second point. Assume that the distance between $\Ocal$ and $\Pcal$ is equal to $\|\rho^H\|$.
We may assume that $\Ocal=K\lambda$ with $\lambda\in \rho^K+\CW$ regular admissible.
We write $\Pcal=K\mu$, with $\mu$ a point in $\tgot^*$, such that $\|\lambda-\mu\|^2=\|\rho^{K_\mu}\|$.
This  implies that $\mu$ belongs to $\CW$, and that $s(\Pcal)=\Ocal$. So $\QS_K(\Pcal)=\QS_K(\Ocal)$. $\Box$
\end{proof}

\begin{exam} Let us go back to Example \ref{exa:ExampleU3} for the group $K=SU(3)$.
We see that the orbit of $\rho^K$
has one ancestor (itself) in the regular sheet, two ancestors in the subregular sheet, and one ancestor $0$ in the
sheet $\{0\}$.
\end{exam}

In general an orbit $\Ocal =K\mu$ has only one ancestor, that is itself.
Only the orbits $\Ocal$ belonging to the boundary of the shifted Weyl chamber might have lower dimensional ancestors. For example the orbits $\Pcal_\sigma$ of the orthogonal projections of $\rho^K$ on the $2^r$
linear spaces $\R \sigma$ ($\sigma\in \Fcal_\kgot$) are ancestors of $o(\kgot)$.
For all these orbits $\Pcal_\sigma$, the representation $\QS_K(\Pcal_\sigma)$
is the trivial representation of $K$.
As the number of sheets is usually  less than $2^r$, some of the ancestors of $\rho^K$ lie in the same sheet.

%
%
%
%
%
%
%
%
%

\bigskip

Finally we end this article by an induction formula relating  $H$-admissible coadjoint orbits to $K$-admissible coadjoint orbits.

Consider the open subset $\hgot^*_0:=\{\xi\in\hgot^*\,\vert\, K_\xi\subset H\}$.
Equivalently, the element $\xi$, identified to an element of $\hgot$, is such that the transformation
$\ad (\xi)$ is invertible on $\qgot:=\kgot/\hgot$, so it determines a complex structure $J_\xi$ on $\qgot$.
We see that the complex structure $J_\xi$ depends only of the connected component
$C$ of $\hgot^*_0$ containing $\xi$.  We denote it by $J_C$ (remark that $J_C$ is $H$-invariant).
We denote $\qgot^C$  the complex $H$-module $(\qgot,J_C)$, and $\rho_C$ the element of
$\zgot^*$ defined by the relation
$$
\langle\rho_C,X\rangle= \frac{1}{2i}\tr_{\qgot^C}\mathrm{ad}(X), \quad X\in\hgot.
$$

We define the holomorphic induction map $\mathrm{Hol}_{H}^{^K}: R(H)\to R(K)$ by the relation
$$
\mathrm{Hol}_{H}^K(V)=\mathrm{Ind}_{H}^{K}\left({\bigwedge}^{\bullet} \qgot^C \otimes V\right).
$$

The following proposition explains the interaction between the holomorphic induction map
$\mathrm{Hol}_{H}^{K} $ and the spin quantization procedure.

\begin{prop}\label{propinduction}
Let $H\mu$ be an admissible orbit for $H$.

$\bullet$ If $\mu+\rho_C\notin \hgot^*_0$, then $\mathrm{Hol}_{H}^{K}(\QS_H(H\mu))=0$.

$\bullet$ If $\mu+\rho_C\in \hgot^*_0$, then $\mu+\rho_C$ is $K$-admissible and
$$
\mathrm{Hol}_{H}^{K}(\QS_H(H\mu))=\epsilon_{C'}^C\,\QS_K(K(\mu+\rho_C))
$$
where $C'$
is the connected component of $\hgot^*_0$ containing $\mu+\rho_C$, and $\epsilon_{C'}^{C}$ is the ratio of
the orientation $o(J_C)$ and $o(J_{C'})$ on $\qgot$.
\end{prop}

\begin{proof}
Let $I_\mu^C=\mathrm{Ind}_{H}^{K}\left(\bigwedge^{\bullet} (\kgot/\hgot)^C \otimes \QS_H(H\mu)\right)$.
 By definition\break $\QS_H(H\mu)=
 \mathrm{Ind}_{H_\mu}^H\Big(\bigwedge^{\bullet} (\hgot/\hgot_\mu)^\mu\otimes \C_{\mu-\rho^H(\mu)}\Big)$.

 Assume first that $\mu':=\mu+\rho_C\in \hgot^*_0$ : let $C'$ be the connected component of
 $\hgot^*_0$ containing $\mu'$. As $K_{\mu'}=H_{\mu'}=H_\mu$, we have
$$
I_\mu^C=\mathrm{Ind}_{K_{\mu'}}^K\Big({\bigwedge}^{\bullet} (\kgot/\hgot)^C \otimes
{\bigwedge}^{\bullet} (\hgot/\hgot_{\mu'})^{\mu'}\otimes \C_{\mu'-\rho_C-\rho^H(\mu')} \Big).
$$
Now we use the fact that the graded $K_{\mu'}$-module $\bigwedge^{\bullet} (\kgot/\hgot)^C$ is equal to
$\epsilon^C_{C'}\bigwedge^{\bullet} (\kgot/\hgot)^{C'}\otimes \C_{\rho_C-\rho_{C'}}$. It gives
that
\begin{eqnarray*}
I_\mu^C&=&\epsilon^C_{C'}\,\mathrm{Ind}_{K_{\mu'}}^K\Big({\bigwedge}^{\bullet} (\kgot/\hgot)^{C'}\otimes
{\bigwedge}^{\bullet} (\hgot/\hgot_{\mu'})^{\mu'}\otimes \C_{\mu'-\rho_{C'}-\rho^H(\mu')} \Big)\\
&=&\epsilon^C_{C'}\,\mathrm{Ind}_{K_{\mu'}}^K\Big(
{\bigwedge}^{\bullet} (\kgot/\kgot_{\mu'})^{\mu'}\otimes \C_{\mu'-\rho(\mu')} \Big)\\
&=& \epsilon^C_{C'}\,\QS_K(K\mu').
\end{eqnarray*}
Here we have used that $\rho(\mu')=\rho_{C'}+\rho^H(\mu')$.

Assume now that $I_\mu^C\neq 0$. Thus $\QS_H(H\mu)$ must be non zero. Hence we have
$\QS_H(H\mu)=\QS_H(H\tilde{\mu})$ where $\tilde{\mu}\in \mu +o(\hgot_\mu)$ is an $H$-admissible and $H$-regular element.

Consider the maximal torus $T:=H_{\tilde{\mu}}$, and a Weyl chamber $\Ccal=\tgot^*_{\geq 0}$ for $K$
containing $\tilde{\mu}$. Let $J_\Ccal$ be the corresponding complex structure on $\kgot/\tgot$.
Let $\rho^K$ be the $\rho$-element associated to the choice of Weyl chamber. Let $C'$ be the connected component of
$\hgot_0^*$ that contains the open face $\tgot^*_{> 0}$.

If we use the relation $\rho^K=\rho_{C'}+ \rho^H(\tilde{\mu})$, one has like before
\begin{eqnarray*}
I_\mu^C
&=&\mathrm{Ind}_H^K\left({\bigwedge}^{\bullet} (\kgot/\hgot)^C\otimes \QS_H(H\tilde{\mu})\right)\\
&=&\mathrm{Ind}_{T}^K\Big({\bigwedge}^{\bullet} (\kgot/\hgot)^{C}\otimes
{\bigwedge}^{\bullet} (\hgot/\tgot)^{\tilde \mu}\otimes \C_{\tilde{\mu}-\rho^H(\tilde{\mu})} \Big)\\
&=&\epsilon^C_{C'}\,\mathrm{Ind}_{T}^K\Big(
{\bigwedge}^{\bullet} (\kgot/\tgot)^{\Ccal}\otimes \C_{\tilde{\mu}+\rho_C-\rho^K} \Big).
\end{eqnarray*}

We see then that $I_\mu^C\neq 0$ only if $\lambda:=\tilde{\mu}+\rho_C= \mu'+\rho^{H_{\mu'}}$ is a $K$-regular element.

Here we have  $\|\rho^{H_{\mu'}}\|= \|\lambda-\mu'\|$, and on the other hand by the magical inequality we must have
$\|\lambda-\mu'\|\geq \|\rho^{K_{\mu'}}\|$ since $\lambda$ is $K$-regular and admissible.
It forces $\|\rho^{K_{\mu'}}\|$ to be equal to $\|\rho^{H_{\mu'}}\|$, and then $K_{\mu'}=H_{\mu'}$ :
the element $\mu'=\mu+\rho_C$ belongs to $\hgot^*_0$.

The proof is completed. $\Box$
\end{proof}

\medskip

Proposition \ref{propinduction} is a very special case of the formula for equivariant indices of twisted Dirac operators obtained in
\cite{pep-vergne:cras}. Indeed the representation $\mathrm{Hol}_{H}^{K}(\QS_H(H\mu))$ is the equivariant index for a $\spinc$-bundle on $M=K/H_\mu$.
The infinitesimal stabilizer $\kgot_M$ is the conjugacy classes of $\hgot_\mu=\kgot_{\mu+\rho_C}$.
It is indeed a representation associated to the Dixmier sheet attached to   $(\kgot_{\mu+\rho_C})$.


\end{document}